\documentclass[12pt]{amsart}

\usepackage{amssymb}
\usepackage{amsmath}
\usepackage{graphicx}
\allowdisplaybreaks[4]

\theoremstyle{remark}

\newtheorem{para}{\bf}[subsection]

\newtheorem{rem}[para]{\bf Remark}
\theoremstyle{definition}

\newtheorem{dfn}[para]{Definition}

\theoremstyle{plain}

\newtheorem{thm}[para]{Theorem}
\newtheorem{lem}[para]{Lemma}
\newtheorem{cor}[para]{Corollary}
\newtheorem{prop}[para]{Proposition}

\newenvironment{numequation}{\addtocounter{para}{1}
\begin{equation}}{\end{equation}}

\newcommand{\bbF}{{\mathbb F}}

\newcommand{\bbP}{{\mathbb P}}

\newcommand{\fro}{{\mathfrak o}}

\begin{document}


\title{Domains of injectivity for the Gross Hopkins Period Map}





\begin{abstract}
We determine the domain of injectivity of the Gross-Hopkins Period map around each points in the deformation space for a fixed formal module $\bar{F}$ of height 2 that defined over a finite field. And then we will use this to conclude some local analyticity result of the group action for the automorphism group of $\bar{F}$ on the deformation space.  
\end{abstract}

\maketitle


\tableofcontents



\section{Introduction}

In this paper, we will show if we fixed $\gamma>0$ and for any $u_0$ in the deformation $X$ s.t. $|u_0|=\gamma$, then the Gross-Hopkins Period map $\Phi$ is injective in $\{u:|u-u_0|<(|\pi| \gamma^{-2})^{\frac{1}{q-1}}\}$. More general, the distance between points in the same fiber $\Phi^{-1}(\Phi(u_0))$ is determined and only depends on the norm $\gamma$. The proof bases on the relation between quasi-isogenies of a fixed lifting of a fixed formal module $\bar{F}$. Then we can describe the fiber of the period map and in particular the conclusion on domain of injectivity. In section 3, we will discuss the image of the domain of injectivity. We will end the paper with the discussion on the local analyticity using the result of \cite{Ch} along with the domain of injectivity.\\

{\it Notation}: $p$ is a prime and $q$ is a power of $p$. $K$ is an finite extension of $\mathbb{Q}_p$ with residue field $\mathbb{F}_q$ and uniformizer $\pi$. Denote $A$ to be its ring of integers. The completion of the algebraic closure of $\mathbb{Q}_p$ is denoted by $\mathbb{C}_p$ with ring of integers $\fro_{\mathbb{C}_p}$ and the valuation ideal $\mathfrak{m}_{\mathbb{C}_p}$. \\

A formal $A-$module law $(F,[\cdot]_F)$ over a $A-$algebra $R$ is:
\begin{enumerate}
	\item $F(X,Y)$ is a formal power series in two variables $X,Y$ over $R$ that satisfies: $F(X,Y)=F(Y,X)$(commutative), $F(X,F(Y,Z))=F(F(X,Y),Z)$(associative) and $F(X,Y)=X+Y+$higher order terms.  
	\item $[\cdot]_F:A\rightarrow R[[X]]$ such that for each $a\in A$, the power series $[a]_F$ is a homomorphism, that is, $[a]_F(F(X,Y))=F([a]_F(X),[a]_F(Y))$ and $[a]_F(X)=a\cdot 1 X+$higher order terms where $a\cdot 1$ is the image of $a$ in $R$.\\ 
\end{enumerate}

$(\bar{F},[\cdot]_{\bar{F}})$ is the height 2 formal $A$ module law over $\bbF_q$ with $[\zeta]_{\bar{F}}(x)=\zeta x$ for $\zeta\in \mu_{q-1}$. Here $\mu_{n}$ denotes the roots of unity of order divides $n$. $K_2$ is the degree 2 unramified extension of $K$ with its ring of integers $A_2$. $D$ is the division algebra 
with center $K$ with invariant $1/2$ and $\fro_D$ is it ring of integers. $G$ is the automorphism group of $\bar{F}$ which is isomorphic to $\fro_D^*$ and has an induced action on the deformation space $\mathfrak{X}=spf(A[[u]])$ of $\bar{F}$. There is a universal deformation $F_u$ over $A[[u]]$ such that any deformation $F$ of $\bar{F}$ over a complete noetherian local $A$-algebra $R$, is $*$-isomorphic to the push forward of $F_u$ under an unique map $A[[u]]\rightarrow R$. In \cite{GH}, there is a $G-$equivariant morphism $\Phi=[\phi_0:\phi_1]$ from $X=\mathfrak{X}^{\textrm{rig}}=\mathfrak{X}\otimes K$ to $(\mathbb{P}^1)^{\textrm{rig}}$ and explicit formula for the coordinate functions $\phi_0$ and $\phi_1$ is given there. We will denote $X_\gamma=\{u\in X(\mathbb{C}_p)\,\,:\,\,|u|\leq \gamma \}$, $\partial X_\gamma=\{u\in X(\mathbb{C}_p)\,\,:\,\,|u|= \gamma \}$ and $X_\gamma^\circ=\{u\in X(\mathbb{C}_p)\,\,:\,\,|u|< \gamma \}$ for $\gamma>0$.\\

{\it Terminology} In this paper we often refer to a formal $A$-module law simply by formal $A$-module.



\section{Domains of Injectivity}
\subsection{The special case of quasi-canonical liftings}
\begin{para}
In this paragraph, let $\mathcal{O}$ be the ring of integers for the unramified field extension $K[\zeta_{q^2-1}]$ of $K$ where $\zeta_{q^2-1}$ is an primitive $q^2-1$ root of unity.\\
Let $M$ be the completion of the maximal unramified field extension and $W$ be its ring of integers. \\

From \cite[prop. 5.3]{Gro}, we know that if $u_0$ is a quasi-canonical lifting of level $s>0$, then $u_0$ is an uniformizer of the ring of integers $W'$ of an abelian extension of $M'/M$. The Galois group $Gal(M'/M)\simeq (\mathcal{O}/\pi^s\mathcal{O})^\times/(A/\pi^s A)^\times$ acts simply transitively on the quasi-canonical lifting of level $s$. In particular $|u_0|=|\pi|^{\frac{1}{q^s+q^{s-1}}}$.\\

The (upper numbering) ramifications subgroups of $Gal(M'/M)$ are given by:
$$Gal(M'/M)=G^0 \supsetneq G^1 \supsetneq \dots \supsetneq G^s=\{\,\,e\,\,\}$$ where
$G^i\simeq (A+\pi^{i}\mathcal{O}/\pi^s\mathcal{O})^\times/(A/\pi^s A)^\times$. \\
This is because we can find a finite abelian extension $M''/M$ generated by torsion points $\tilde{F}[\pi^s]=\{\alpha\in\mathfrak{m}_{\mathbb{C}_p}  \,\,|\,\,[\pi^n](\alpha)=0 \}$ of a height 1 formal $\mathcal{O}-$module $\tilde{F}$, where $\tilde{F}$ is a lifting of $\bar{F}$ over $\mathcal{O}$. Then from \cite[ch. III, sec. 8]{Neu}, we get the ramification subgroups of $Gal(M''/M)^i\simeq (1+\pi^i\mathcal{O})/(1+\pi^s\mathcal{O})$ for $1\leq i\leq s$. Since $M'$ is a subextension with $Gal(M''/M')\simeq (A+\pi^s\mathcal{O}/\pi^s\mathcal{O})^\times$, we get the corresponding ramification subgroups. \\

Define $\psi(v)=\int^v_0 [G^0:G^t]dt$.\\

For integers $0\leq n\leq s$,
$$\psi(n)=\sum_{i=1}^{n} (q+1)q^{i-1}=\frac{(q+1)(q^n-1)}{q-1}$$.\\

So the (lower numbering) ramifications subgroups are
$$Gal(M'/M)=G_0 \supsetneq G_{q+1} \supsetneq \dots \supsetneq G_{\frac{(q+1)(q^s-1)}{q-1}}=\{\,\,e\,\,\}.$$

If $g\in G_{\frac{(q+1)(q^n-1)}{q-1}}\backslash G_{\frac{(q+1)(q^{n+1}-1)}{q-1}}$, then
$$|g\cdot u_0-u_0|=|u_0|^{1+\frac{(q+1)(q^n-1)}{q-1}}=|u_0|^{\frac{q^{n+1}+q^{n}-2}{q-1}}=|\pi|^{\frac{q^{n+1}+q^{n}-2}{q^{s-1}(q^2-1)}}$$
for $0 \leq n \leq s-1$.\\

So there are exactly $q-1$ quasi-canonical liftings of level $s$ contained in $\{\,\,u \,:\, |u-u_0|=|\pi|^{\frac{q^{s}+q^{s-1}-2}{q^{s-1}(q^2-1)}}\,\,\}$;
there are exactly $q^2-q$ quasi-canonical liftings of level $s$ contained in $\{\,\,u \,:\, |u-u_0|=|\pi|^{\frac{q^{s-1}+q^{s-2}-2}{q^{s-1}(q^2-1)}}\,\,\}$; \dots ;
there are exactly $q^{s-1}-q^{s-2}$ quasi-canonical liftings of level $s$ contained in $\{\,\,u \,:\, |u-u_0|=|\pi|^{\frac{q^{2}+q^{1}-2}{q^{s-1}(q^2-1)}}\,\,\}$, and
there are exactly $q^s$ quasi-canonical liftings of level $s$ contained in $\{\,\,u \,:\, |u-u_0|=|\pi|^{\frac{1}{q^{s-1}(q+1)}}\,\,\}$.\\

\end{para}

\begin{para}
Now in this paragraph, let $\mathcal{O}$ be the ring of integers for the ramified field extension $K[\sqrt{\pi}]$ of $K$ where $\sqrt{\pi}$ is one of the square root of $\pi$.\\
Let $M$ be the completion of the maximal unramified field extension of $K[\sqrt{\pi}]$ and $W$ be its ring of integers. \\

Similarly, from \cite[prop. 5.3]{Gro}, we know that if $u_0$ is a quasi-canonical lifting of level $s>0$, then $u_0$ is an uniformizer of the ring of integers $W'$ of an abelian extension of $M'/M$. The Galois group $Gal(M'/M)\simeq (\mathcal{O}/\pi^s\mathcal{O})^\times/(A/\pi^s A)^\times$ acts simply transitively on the quasi-canonical lifting of level $s$. In particular $|u_0|=|\sqrt{\pi}|^{\frac{1}{q^s}}=|\pi|^{\frac{1}{2q^s}}$.\\

The (upper numbering) ramifications subgroups of $Gal(M'/M)$ are given by:
$$Gal(M'/M)=G^0=G^1 \supsetneq G^2=G^3 \supsetneq \dots \supsetneq G^{2s-2}=G^{2s-1}\supsetneq G^{2s}=\{\,\,e\,\,\}$$ where
$G^{2i}\simeq (A+\pi^{i}\mathcal{O}/\pi^s\mathcal{O})^\times/(A/\pi^s A)^\times$. \\
This is because we can find a finite abelian extension $M''/M$ generated by torsion points $\tilde{F}[\sqrt{\pi}^{2s}]$ of a height 1 formal $\mathcal{O}-$module $\tilde{F}$, where $\tilde{F}$ is a lifting of $F$ over $\mathcal{O}$. Then from \cite[ch. III, sec. 8]{Neu}, we get the ramification subgroups of $Gal(M''/M)^j\simeq (1+\sqrt{\pi}^j\mathcal{O})/(1+\pi^s\mathcal{O})$ for $1\leq j\leq 2s$. Since $M'$ is a subextension with $Gal(M''/M')\simeq (A+\pi^s\mathcal{O}/\pi^s\mathcal{O})^\times$, we get the corresponding ramification subgroups. \\

Define $\psi(v)=\int^v_0 [G^0:G^t]dt$.\\

For integers $1\leq n\leq s$,
$$\psi(2n-1)=1+\sum_{i=2}^{n} 2q^{i-1}=2\frac{(q^n-1)}{q-1}-1$$.\\

So the (lower numbering) ramifications subgroups are
$$Gal(M'/M)=G_1 \supsetneq G_{2q+1} \supsetneq \dots \supsetneq G_{2\frac{(q^s-1)}{q-1}-1}\supsetneq \{\,\,e\,\,\}.$$

If $g\in G_{2\frac{(q^n-1)}{q-1}-1}\backslash G_{2\frac{(q^{n+1}-1)}{q-1}-1}$, then
$$|g\cdot u_0-u_0|=|u_0|^{2\frac{(q^n-1)}{q-1}}=|\pi|^{\frac{q^{n}-1}{q^{s}(q-1)}}$$
for $1 \leq n \leq s$.\\

So there are exactly $q-1$ quasi-canonical liftings of level $s$ contained in $\{\,\,u \,:\, |u-u_1|=|\pi|^{\frac{q^{s}-1}{q^{s}(q-1)}}\,\,\}$;
there are exactly $q^2-q$ quasi-canonical liftings of level $s$ contained in $\{\,\,u \,:\, |u-u_1|=|\pi|^{\frac{q^{s-1}-1}{q^{s}(q-1)}}\,\,\}$; \dots , and
there are exactly $q^{s}-q^{s-1}$ quasi-canonical liftings of level $s$ contained in $\{\,\,u \,:\, |u-u_1|=|\pi|^{\frac{1}{q^{s}}}\,\,\}$.\\

\end{para}

\subsection{Facts from Lubin's papers}

\begin{para}
\label{PaLu}
In this section, We will need argument in Lubin's paper \cite{Lu1} .\\

Let's begin with recalling some results in \cite{Lu2}. For any deformation $(F,[\cdot])$ of $(\bar{F},[\cdot]_{\bar{F}})$ over a complete local noetherian $A$-algebra $R$ inside $\fro_{\mathbb{C}_p}$, the torsions submodules $F[\pi^n]=\{\alpha\in\mathfrak{m}_{\mathbb{C}_p}  \,\,|\,\,[\pi^n](\alpha)=0 \}$ is a free $A/\pi^n$ of rank 2. If $C$ is a submodule of $F[\pi^n]$ for some $n$, then define $f(x)=\prod_{\alpha\in C}F(x,\alpha)$ and there exists another formal $A$ module $(F_C,[\cdot]_C)$ over $R[[\alpha]]_{\alpha\in C}$ s.t. this $f$ is the formal $A$ module homomorphism from $(F,[\cdot])$ to $(F_C,[\cdot]_C)$ with kernel $C$. We called that $f$ is a quasi-isogeny of height $m$ if $C$ has $q^m$ elements. Quasi-isogeny induce an equivalence relation on the deformation space. For example, there is another quasi-isogeny $g:(F_C,[\cdot]_C)\rightarrow (F,[\cdot])$ of height $m$ s.t. $g\circ f=[\pi]_F$ and $f\circ g=[\pi]_{F_C}$.\\

Suppose $u_0\in X(\mathbb{C}_p)$ with $|u_0|=\gamma$. Consider the formal typical $A$ module $(F_{u_0},[\cdot]_{u_0}$) over $A[[u_0]]$ as a deformation of $\bar{F}$ corresponding to $u_0$. Since $F_{u_0}[\pi]$ is 2-dimensional $\mathbb{F}_q=A/\pi$ vector space, there are $q+1$ distinct subspaces, denoted them by $C_i$, $i=1,2,\dots,q+1$. Define $(F_i,[\cdot]_i)$ the formal module and $f_i=\prod_{\alpha\in C_i}F(x,\alpha)$ s.t. $f_i$ is a quasi-isogeny of height 1. Notice that $[\zeta]_i(x)=\zeta x$ as $[\zeta]_{u_0}=\zeta x$ and $f_i(\zeta x)=\zeta f_i(x)$ for $\zeta$ s.t. $\zeta^{q-1}=1$. As a consequence, $[\pi]_i(x)$ only has terms in degrees of the form $1+n(q-1)$ for some $n\geq 0$. \\ 

For each $F_i$, there is a $u_i\in X(\mathbb{C}_p)$ with corresponding deformation $(F_{u_i},[\cdot]_{u_i}$) s.t. $(F_{u_i},[\cdot]_{u_i}$) is strictly isomorphic to $(F_i,[\cdot]_i)$ by an unique isomorphism $\theta_i$, see \cite{GH} prop 12.10 and \cite{LT} prop 3.1. So we have the quasi-isogeny $\theta_i\circ f_i:(F_{u_0},[\cdot]_{u_0})\rightarrow (F_{u_i},[\cdot]_{u_i})$.\\


\begin{dfn}
\label{PaCh}
We will simply call $u_i$ and $u_0$ are corresponds (by a quasi-isogeny) of height 1.\\
\end{dfn}

If we look closer to $u_i$ given by a $u_0$, there are two possibilities, for details, see \cite{Lu1}:\\
(Case 1) If $|u_0|=\gamma\leq |\pi|^{\frac{q}{q+1}}$, then all non zero elements in $F_{u_0}[\pi]$ has norm $|\pi|^{\frac{1}{q^2-q}}$ and all $u_i$ has norm $|\pi|^{\frac{1}{q+1}}$.\\
(Case 2) If $\gamma> |\pi|^{\frac{q}{q+1}}$, then $q^2-q$ elements in $F_{u_0}[\pi]$ has norm $\gamma^{\frac{1}{q^2-q}}$ and $q-1$ elements has smaller norm $(\frac{|\pi|}{\gamma})^{\frac{1}{q-1}}$. We can choose the generator $\beta$ of $C_{q+1}$ to be the one has norm $(\frac{|\pi|}{\gamma})^{\frac{1}{q-1}}$. Then $|u_{q+1}|<|u_i|=\gamma^{\frac{1}{q}}$ for $i\neq q+1$.\\

We can always factor $f$ into the composition of $n$ quasi-isogenies of height 1.\\
\end{para}

\subsection{Period map and quasi-isogeny}

\begin{para}
Under the period map $\Phi$, all $u_i$, $i=1,2,\dots,q+1$ that correspond to the same $u_0$ has the same image. Indeed, we have the following proposition:
\begin{prop}
Suppose $u_0,u_1\in X(\mathbb{C}_p)$ and $u_1$ correspond to $u_0$. Then 
\[[\phi_0(u_1),\phi_1(u_1)]=[\pi\phi_1(u_0),\phi_0(u_0)].\]
\end{prop}

\begin{proof}
The argument is similar to that  of proving the vector bundle $\mathcal{L}ie(E_u)$ of the deformation space is generically flat where $E_u$ is the Universal additive extension of the universal deformation $F_u$. Cf \cite{GH} section 22. We have an exact sequence of formal $A$-modules:
\[0\rightarrow \mathbb{G}_a \otimes Ext(F_u,\mathbb{G}_a)\rightarrow E_u \rightarrow F_u\rightarrow 0\]
and hence an exact sequence of free $A[[u]]$-modules:
\[0\rightarrow Lie(\mathbb{G}_a \otimes Ext(F_u,\mathbb{G}_a))\stackrel{\psi}{\rightarrow} Lie(E_u)=Hom(RigExt(F_u,\mathbb{G}_a),A[[u]]) \stackrel{\varphi}{\rightarrow}Lie(F_u)\rightarrow 0.\] 

From \cite{GH} Section 21, there exists $c_0,c_1\in Hom(RigExt(F_u,\mathbb{G}_a),S)$\\
$=H^0(X\otimes K_2,\mathcal{L}ie(E_u)\otimes_{A[[u]]}(\mathcal{O}_X\otimes K_2))$ that forms a basis over $S=H^0(X\otimes K_2,\mathcal{O}_X\otimes K_2)$. Let's recall the definition of $c_0$ and $c_1$. Here $RigExt(F_u,\mathbb{G}_a)$ is a free $A[[u]]$-module of rank 2 with basis represented by quasi-logarithm $g_0(u,x),g_1(u,x)\in S[[x]]$. The (quasi-)logarithm $g_0(u,x)$ satisfies the functional equation \[g_0(u,x)=x+\frac{u}{\pi}g_0(u^q,x^q)+\frac{1}{\pi}g_0(u^{q^2},x^{q^2})\]
and $g_1(u,x)=\frac{\partial}{\partial u}g_0(u,x)$. Any quasi-logarithm of $F_u$ is of the form $g_u(x)=g(u,x)=\sum_{k\geq 0} m_k(u)x^{q^k}$ where $m_k(u)\in S$. And 
\[c_0(g)(u):=\lim \pi^k m_{2k}(u)\] and \[c_1(g)(u):=\lim \pi^k m_{2k-1}(u).\]

In particular, $c_i(g_j)(0)=\delta_{ij}$ for $0\leq i,j\leq 1$.\\

The period map is defined by $u\mapsto [\phi_0(u),\phi_1(u)]$ where $\phi_i(u)=c_i(g_0)(u)$. The vector $(\phi_0(u),\phi_1(u))$ here is the normal vector to the subspace given by the image of $\psi$. Indeed, if  we identify $Lie(F_u)$ with $A[[u]]$, $\varphi(c)=c(g_0)$ where $c\in Lie(E_{u})=Hom(RigExt(F_u,\mathbb{G}_a),A[[u]])$. And thus $\psi\otimes S(Lie(\mathbb{G}_a\otimes Ext(F_{u},\mathbb{G}_a))\otimes S)=\{a_0c_0+a_1c_1\in Lie(E_u)\otimes T\,\,:\,\, a_0\phi_0(u)+a_1\phi_1(u)=0\}$.\\

Suppose $f$ is the quasi-isogeny of height 1 from $F_{u_0}$ to $F_{u_1}$ that defined over $T$ where $T$ is a finite extension of $\widehat{K_2(u_0,u_1)}$, then we have the following commutative diagram
\begin{numequation}
\label{eq:CDoUE}
\begin{array}{ccccccccc}
0 &\rightarrow & Lie(\mathbb{G}_a\otimes Ext(F_{u_0},\mathbb{G}_a))\otimes T &\rightarrow & Lie(E_{u_0})\otimes T& \rightarrow  & Lie(F_{u_0})\otimes T  &\rightarrow & 0\\
& & \downarrow & & \downarrow & & \downarrow& & \\	
0 &\rightarrow & Lie(\mathbb{G}_a\otimes Ext(F_{u_1},\mathbb{G}_a))\otimes T &\rightarrow & Lie(E_{u_1})\otimes T &\rightarrow  & Lie(F_{u_1})\otimes T  &\rightarrow & 0\\
\end{array}
\end{numequation} where the vertical maps are induced from $f$.\\

Here the middle vertical map sends $c_i$ to $\tilde{c}_i$ where $\tilde{c}_i(g)(u_1)=c_i(g\circ(f))(u_1)$. It follows from the commutative diagram \ref{eq:CDoUE} that the vector $(\tilde{c}_0(g_0)(u_1),\tilde{c}_1(g_0)(u_1))=C(\phi_0(u_0),\phi_1(u_0))$ for some constant $C\in T$.\\

Now we need to compute $\tilde{c}_i(g_0)(u_1)=c_i(g_0\circ f)(u_1)$.\\
Because the coefficients of $f(x)-x^q$ have norm strictly less than 1. Then Prop 22.2 in \cite{GH} implies $\tilde{c}_i(g_0)(u_1)=c_i(g_0\circ(x^q))(u_1)$. It is not hard to check that $c_0(g_0\circ(x^q))(u_1)=\pi c_1(g_0)(u_1)$ and $c_1(g_0\circ(x^q))(u_1)=c_0(g_0)(u_1)$. we can then conclude that $C\neq 0$ and hence the result.\\

\end{proof}

Hence if we look at a fiber $\Phi^{-1}(w)$ restricted to an annulus $\{u\in X(\mathbb{C}_p)\,\,|\,\, |u|=\gamma\}$, denoted by $\partial X_\gamma$, then they all come from a single $u_0$ with $|u_0|<|\pi|^{\frac{1}{q+1}}$ in the following sense:\\

For $u_1\in \Phi^{-1}(w)\cap \partial X_\gamma$, there exists a free $A/\pi^n$ rank 1 submodule $C$ in $F_{u_0}[\pi^n]$, where $n$ is the smallest integer with $\gamma^{q^n}<|\pi|^{\frac{1}{q+1}}$, s.t. $C$ is the kernel for a quasi-isogeny $f:(F_{u_0},[\cdot]_{u_0})\rightarrow (F_{u_1},[\cdot]_{u_1})$ of height $n$. \\

\end{para}

\subsection{Main result for fiber of the period map}
\begin{para}
Let me state the main result:
\begin{thm}
\label{ThmInj}
\begin{enumerate}
	\item Suppose $\gamma$ is not of the form $|\pi|^{\frac{1}{q^s+q^{s+1}}}$ for $s\geq 0$ and $u_1\in \partial X_\gamma$. Then we can order the element in $\Phi^{-1}(\Phi(u_1))\cap \partial X_\gamma$ as $u_1,u_2,\dots,u_{q^n}$ where $n$ is the smallest non-negative integer with $\gamma^{q^n}<|\pi|^{\frac{1}{q+1}}$ such that 
	\[|u_i-u_1|=(\frac{|\pi|}{\gamma^{2q^{j-1}}})^{\frac{1}{q^{j}-q^{j-1}}}\] for $1\leq j\leq n$ and $q^{j-1}<i\leq q^{j}$.\\
	\item Suppose $\gamma=|\pi|^{\frac{1}{q^s+q^{s+1}}}$ for some $s\geq 0$ and $u_1\in \partial X_\gamma$. Then we can order the element in $\Phi^{-1}(\Phi(u_1))\cap \partial X_\gamma$ as $u_1,u_2,\dots,u_{q^s+q^{s+1}}$ such that 
	\[|u_i-u_1|=(\frac{|\pi|}{\gamma^{2q^{j-1}}})^{\frac{1}{q^{j}-q^{j-1}}}\] for $1\leq j\leq s$ and $q^{j-1}<i\leq q^{j}$ and $|u_i-u_1|=\gamma(=(\frac{|\pi|}{\gamma^{2q^{s}}})^{\frac{1}{q^{s+1}-q^{s}}})$ for $i>q^s$.\\
\end{enumerate}
\end{thm}

We need several lemmas before proving the theorem.

\begin{lem}
\label{LeS1}
Suppose $|u_0|=\gamma$ and $u_i$, for $i=1,2,\dots,q+1$ corresponds to $u_0$ of height 1.
\begin{enumerate}
	\item Suppose $\gamma>|\pi|^{\frac{q}{q+1}}$. We assume $u_{q+1}$ is the only one that has smaller norm. Then $|u_i-u_j|\leq |\pi|^{\frac{1}{q-1}}\gamma^{\frac{-2}{q^2-q}}$ for $1\leq i,j \leq q$.\\
	\item Suppose $|u_0|=\gamma\leq |\pi|^{\frac{q}{q+1}}$, then $|u_i-u_j|\leq |\pi|^{\frac{1}{q+1}}$ $\forall\,\,i,j$.\\
\end{enumerate}
\end{lem} 

\begin{proof}
\begin{enumerate}
	\item
	Assume $\alpha_i$ is the generator of $C_i$ as in paragraph \ref{PaLu}. Then $|\alpha_i|>|\alpha_{q+1}|$ for $i\neq q+1$. For distinct $i,j$ that are both less than q+1, $\alpha_i=F_{u_0}(\xi \alpha_j,\zeta \alpha_{q+1})$ for some $\xi,\zeta\in \mu_{q-1}$. Without loss of generality, assume $\xi=1$. Then $|\alpha_i-\alpha_j|=|\alpha_{q+1}|$ as $F_{u_0}(x,y)-x-y$ is divisible by $xy$. Therefore, $|(\alpha_i)^{q-1}-(\alpha_j)^{q-1}|=|\alpha_i|^{q-2}|\alpha_i-\alpha_j|=(\gamma^{\frac{1}{q^2-q}})^{q-2}((\frac{|\pi|}{\gamma})^{\frac{1}{q-1}})=|\pi|^{\frac{1}{q-1}}\gamma^{\frac{-2}{q^2-q}}$.\\

	By Lemma \ref{LeHo} below, we have $f_i\equiv f_j $ mod $(\alpha_1)^{q-2}(\alpha_{q+1})\fro_{\mathbb{C}_p}$\\
as $f_i(x)=\prod_{\alpha\in C_i}F(x,\alpha)=x(\prod^{q-1}_{l=1} F_{u_0}(x,(\zeta_{q-1})^{l}\alpha_i))$.\\
 
	Since $f_i$ are homomorphism, $[a]_i\circ f_i=f_i\circ [a]_{u_0}$ where $a\in A$ and hence modulo $(\alpha_1)^{q-2}(\alpha_{q+1})\fro_{\mathbb{C}_p}$, we have
\[ [a]_i\circ f_i=f_i\circ [a]_{u_0}\equiv f_j\circ [a]_{u_0}= [a]_j\circ f_j \equiv [a]_j \circ f_i.   \]
	Because $f_i$ is distinguish at degree $q$ and that coefficient is unit, we must have $[a]_i\equiv [a]_j$ mod $(\alpha_1)^{q-2}(\alpha_{q+1})\fro_{\mathbb{C}_p}$. Similarly, $F_i(x,y)\equiv F_j(x,y)$ mod $(\alpha_1)^{q-2}(\alpha_{q+1})\fro_{\mathbb{C}_p}$.\\

	Since the universal deformation is functorial, $|u_i-u_j|\leq |\pi|^{\frac{1}{q-1}}\gamma^{\frac{-2}{q^2-q}}$.

	\item It is clear as $|u_i|=|\pi|^{\frac{1}{q+1}}$ for all $i$.\\

\end{enumerate}

\end{proof}

\begin{lem}
\label{LeHo}

Let $g(x,y)=\prod^{q-1}_{i=1}F_{u_0}(x,(\zeta_{q-1})^{i}y)\in \fro_{\mathbb{C}_p}[[x,y]]$.\\
Then 
\begin{enumerate}
	\item $g(x,y)\in \fro_{\mathbb{C}_p}[[x^{q-1},y^{q-1}]]$
	\item $g(x,y)=-g(y,x)$
	\item $g(x,0)=x^{q-1}$
\end{enumerate}
\end{lem}
\begin{proof}
Since $[\zeta_{q-1}]_{u_0}(x)=\zeta_{q-1}x$, $F_{u_0}(\zeta_{q-1}x,\zeta_{q-1}y)=\zeta_{q-1}F_{u_0}(x,y)$. Hence $g(\zeta_{q-1}x,y)=\prod^{q-1}_{i=1}F_{u_0}(\zeta_{q-1}x,(\zeta_{q-1})^{i}y)=(\zeta_{q-1})^{q-1}g(x,y)=g(x,y)$. Hence $g(x,y)\in \fro_{\mathbb{C}_p}[[x^{q-1},y]]$.\\

Next, $g(y,x)=\prod^{q-1}_{i=1}F_{u_0}(y,(\zeta_{q-1})^{i}x)=\prod^{q-1}_{i=1}F_{u_0}((\zeta_{q-1})^{i}x,y)=(\zeta_{q-1})^{\frac{q(q-1)}{2}}g(x,y)=-g(x,y)$.\\

Finally, $g(x,0)=(F_{u_0}(x,0))^{q-1}=x^{q-1}$.\\
\end{proof}

\begin{lem}{Continuity of Hecke correspondence}
\label{LeCoHe}
\begin{enumerate}
	\item Suppose $\gamma>|\pi|^{\frac{q}{q+1}}$ and $u_0,u_0'\in \partial X_\gamma$. Let $u_i$ (respectively $u_i'$) $i=1,2,\dots,q+1$ correspond to $u_0$ (respectively  $u_0'$) of height 1 such that $|u_{q+1}|$ (respectively $|u_{q+1}'|$) is smallest. Assume $|u_0-u_0'|\leq d$.\\
	\begin{enumerate}
		\item If $d<|\pi|^{\frac{q}{q-1}}\gamma^{\frac{-2}{q-1}}$, then for each $u_j'$, $j<q+1$, there is some $u_i$, $i<q+1$ such that $|u_i-u_j'|<d|\pi|^{-1}\gamma^{\frac{2}{q}}$
		\item If $|\pi|^{\frac{q}{q-1}}\gamma^{\frac{-2}{q-1}}\leq d \leq \gamma$, then for each $u_j'$, $j<q+1$, $|u_i-u_j'|\leq d^{\frac{1}{q}}$ for any $i<q+1$.\\
	\end{enumerate}

	\item Suppose $u_0,u_0'\in X_{|\pi|^{\frac{q}{q+1}}}$. Let $u_i$ (respectively $u_i'$) $i=1,2,\dots,q+1$ correspond to $u_0$ (respectively  $u_0'$) of height 1. Assume $|u_0-u_0'|\leq d\leq \gamma$.\\
	Then for any $u_j'$, there exists $u_i$ such that $|u_i-u_j'|<d|\pi|^{\frac{1-q}{q+1}}$.\\

\end{enumerate}
\end{lem}

\begin{proof}
\begin{enumerate}
	\item
	Before we start, let look at the set $F_{u_0}[\pi]$. If $\{\alpha,\beta\}$ is a basis with $|\beta|<|\alpha|$. Then any nonzero element is either of the form $\xi\alpha$, $\zeta\beta$ or $F_{u_0}(\xi\alpha,\zeta\beta)$ where $\zeta,\xi\in \mu_{q-1}$.\\
	We observe that $|\alpha-\xi\alpha|=|\alpha-F_{u_0}(\xi\alpha,\zeta\beta)|=|\alpha-\zeta\beta|=|\alpha|$ if $\xi \neq 1$ and $|\alpha-F_{u_0}(\alpha,\zeta\beta)|=|\beta|$.\\
	
	Suppose $\alpha'\in F_{u_0'}[\pi]$ generates the kernel of the quasi-isogeny from $u_0'$ to $u_j'$, in particular, $|\alpha'|=\gamma^{\frac{1}{q^2-q}}$.	Choose $\alpha\in F_{u_0}[\pi]$ with minimal $|\alpha-\alpha'|$.\\
	
	If $|\alpha-\alpha'|<|\beta|$, then
	\[\left|\prod_{\eta\in F_{u_0}[\pi]}(\alpha'-\eta)\right|=|\alpha-\alpha'||\beta|^{q-1}|\alpha|^{q^2-q}=|\pi||\alpha-\alpha'|.\]

	If $|\beta|\leq|\alpha-\alpha'|\leq |\alpha|$, then 
	\[\left|\prod_{\eta\in F_{u_0}[\pi]}(\alpha'-\eta)\right|=|\alpha-\alpha'|^q|\alpha|^{q^2-q}=\gamma|\alpha-\alpha'|^q.\]
	On the other hand,
	\[\left|\prod_{\eta\in F_{u_0}[\pi]}(\alpha'-\eta)\right|=\left|[\pi]_{u_0}(\alpha')\right|=\left|[\pi]_{u_0}(\alpha')-[\pi]_{u_0'}(\alpha')\right|\leq d |\alpha'|^{q}=d\gamma^{\frac{1}{q-1}}\] 
	as $(u_0-u_0')$ divides every coefficients of $[\pi]_{u_0}(x)-[\pi]_{u_0'}(x)$ and the first non zero term is of degree $q$.\\
	
	\begin{enumerate}
		\item Suppose $d<|\pi|^{\frac{q}{q-1}}\gamma^{\frac{-2}{q-1}}$.\\
		Claim: $|\alpha-\alpha'|\leq d|\pi|^{-1}\gamma^{\frac{1}{q-1}}\,(<|\beta|)$.\\
		If $|\alpha-\alpha'|\geq |\beta|$, then the above steps show $\gamma|\alpha-\alpha'|^{q}\leq d\gamma^{\frac{1}{q-1}}<|\pi|^{\frac{q}{q-1}}\gamma^{\frac{-1}{q-1}}$. Hence $|\alpha-\alpha'|<(\frac{|\pi|}{\gamma})^{\frac{1}{q-1}}=|\beta|$ which is a contradiction.\\
		
		So $|\alpha-\alpha'|<|\beta|$ and so $|\pi||\alpha-\alpha'|\leq d\gamma^{\frac{1}{q-1}}$ and hence $|\alpha-\alpha'|\leq d|\pi|^{-1}\gamma^{\frac{1}{q-1}}$.\\
		
		Furthermore, $|\alpha^{q-1}-(\alpha')^{q-1}|\leq d|\pi|^{-1}\gamma^{\frac{1}{q-1}}|\alpha|^{q-2}=d|\pi|^{-1}\gamma^{\frac{2}{q}}$.\\
		
		Let $\alpha\in F_{u_0}[\pi]$ generates the kernel of the quasi-isogeny from $u_0$ to some $u_i$. Then $f_i(x)=x\prod_{l=1}^{q-1}F_{u_0}(x,(\zeta_{q-1})^l\alpha)$ and $f_j'(x)=x\prod_{l=1}^{q-1}F_{u_0'}(x,(\zeta_{q-1})^l\alpha')$.\\
		
		Since $|u_0-u_0'|\leq d$, $[a]_{u_0}(x)\equiv [a]_{u_0'}(x)$ for $a\in A$ and $F_{u_0}(x,y)\equiv F_{u_0'}(x,y)$ mod $I_d$ where $I_d:=\{a\in \mathbb{C}_p\,\,|\,\,|a|\leq d\}$. Hence,		
		\[\begin{array}{rclclc}
			f_i(x) & \equiv & x\prod_{l=1}^{q-1}F_{u_0'}(x,(\zeta_{q-1})^l\alpha)&  & \text{mod}\,\, I_d &\\
			 & \equiv & f_j(x) & & \text{mod}\,\, (\pi^{-1} \alpha^{2q-2})I_d & \text{by Lemma}\,\, \ref{LeHo}\\
		\end{array}\]
		Because $|\pi^{-1} \alpha^{2q-2}|=|\pi|^{-1}\gamma^{\frac{2}{q}}>1$, so 
		$I_d \subset (\pi^{-1} \alpha^{2q-2})I_d$. Hence $f_i\equiv f_j'$ mod $(\pi^{-1} \alpha^{2q-2})I_d$.\\
		Therefore, 
		\[[a]_{i}\circ f_i=f_i\circ [a]_{u_0}\equiv  f_j'\circ[a]_{u_0'} =[a]_{j'}\circ f_j'\equiv [a]_{j'}\circ f_i\] mod $(\pi^{-1} \alpha^{2q-2})I_d$ for $a\in A$. As $f_i$ is distinguish at degree $q$ with corresponding coefficient is unit, we conclude $[a]_{u_i}\equiv[a]_{u_j'}$ mod $(\pi^{-1} \alpha^{2q-2})I_d$. Similarly for $F_i(x,y)$ and $F_j'(x,y)$. Since the universal deformation is functorial, we have $|u_i-u_j'|\leq d|\pi|^{-1}\gamma^{\frac{2}{q}}$.\\
		
		\item Suppose $|\pi|^{\frac{q}{q-1}}\gamma^{\frac{-2}{q-1}}\leq d \leq \gamma$.\\
		Claim: $|\alpha-\alpha'|\leq d^{\frac{1}{q}}\gamma^{\frac{2-q}{q^2-q}}$.\\
		If $|\alpha-\alpha'|>d^{\frac{1}{q}}\gamma^{\frac{2-q}{q^2-q}}\geq |\pi|^{\frac{1}{q-1}}\gamma^{\frac{-1}{q-1}}=|\beta|$, then $\gamma|\alpha-\alpha'|^q\leq d\gamma^{\frac{1}{q-1}}$ implies $|\alpha-\alpha'|\leq d^{\frac{1}{q}}\gamma^{\frac{2-q}{q^2-q}}$ which is a contradiction.\\
		So $|\alpha-\alpha'|\leq d^{\frac{1}{q}}\gamma^{\frac{2-q}{q^2-q}}$ and $|\alpha^{q-1}-(\alpha')^{q-1}|\leq d^{\frac{1}{q}}\gamma^{\frac{2-q}{q^2-q}}|\alpha|^{q-2}= d^{\frac{1}{q}}$.\\
		Since $d^{\frac{1}{q}}>d$, so the arguments in part 1(a) holds and hence
		$|u_i-u_j'|\leq d^{\frac{1}{q}}$. Notice also that $d^{\frac{1}{q}}\geq |\pi|^{\frac{1}{q-1}}\gamma^{\frac{-2}{q^2-q}}$. Therefore, $|u_j'-u_i|\leq d^{\frac{1}{q}}$ for all $i=1,2,\dots,q$.\\
	\end{enumerate}	
	
	\item Suppose $u_0,u_0'\in X_{|\pi|^{\frac{q}{q+1}}}$. Then every non zero element in $F_{u_0}[\pi]$ has same norm $|\pi|^{\frac{1}{q^2-1}}$ and any two distinct elements has distance $|\pi|^{\frac{1}{q^2-1}}$.\\
	
	Suppose $\alpha'\in F_{u_0'}[\pi]$ generates the kernel of the quasi-isogeny from $u_0'$ to $u_j'$. Choose $\alpha\in F_{u_0}[\pi]$ with minimal $|\alpha-\alpha'|$. Let $\alpha$ generates the kernel of the quasi-isogeny from $u_0$ to some $u_i$. Then 
	\[\left|\prod_{\eta\in F_{u_0}[\pi]}(\alpha'-\eta)\right|=|\alpha-\alpha'||\alpha|^{q^2-1}=|\pi||\alpha-\alpha'|\leq d|\alpha'|^q\]
	as in part (1). \\
	Hence $|\alpha-\alpha'|\leq d|\pi|^{\frac{1+q-q^2}{q^2-1}}$ and $|\alpha^{q-1}-(\alpha')^{q-1}|\leq d|\pi|^{\frac{-1+2q-q^2}{q^2-1}}=d|\pi|^{\frac{1-q}{q+1}}>d$. So the same arguments show $|u_i-u_j'|\leq d|\pi|^{\frac{1-q}{q+1}}$.\\

\end{enumerate}
\end{proof}

\begin{rem}
We may ask whether the arguments apply to $|u_{q+1}-u_{q+1}'|$ in the situation in part (1) of Lemma \ref{LeCoHe}. And in most cases, we can only get $|u_{q+1}-u_{q+1}'|\leq d$ due to the difference $[\pi]_{u_0}\circ f_{q+1}- [\pi]_{u_0'}\circ f_{q+1}'$ is determined by $u_{q+1}-u_{q+1}'$ instead of $f_{q+1}-f_{q+1}'$. The case that the arguments go through is when $|\pi|^{\frac{q}{q+1}}<\gamma \leq |\pi|^{\frac{1}{2}}$ and so $|u_{q+1}-u_{q+1}'|\leq d|\pi|\gamma^{-2}$. Notice that in this case, $|u_{q+1}|=\frac{|\pi|}{|u_0|}=\frac{|\pi|}{\gamma}$.\\
\end{rem}

\begin{proof}[Proof of Theorem \ref{ThmInj}]
Weierstrass Preparation theorem would be useful in the proof, so we need the following definition:  
\begin{dfn}
A power series $\psi(x)=\sum a_n x^n\in \mathbb{C}_p[[x]]$ is called distinguished of degree $N$ over $X_r$ if
\begin{itemize}
	\item $|a_N| r^N=\sup\{\,|a_n|r^n\, , \,r\geq 0\,\} $ 
	\item $\forall \,r>N$, $|a_n|r^n<|a_N| r^N$
\end{itemize}

\end{dfn}
If $\psi(x)$ is distinguished of degree $N$ over $X_r$, then $\psi(x)$ has exactly $N$ roots(including multiplicity) in $X_r$. Furthermore, if the coefficients of $\psi(x)$ are contained in some finite extension of $\mathbb{Q}_p$, then all roots in $X_r$ are algebraic.\\

We will use induction on $n$ to prove the theorem.\\
When $n=0$, the statement holds automatically as $\Phi$ is injective on $\{|u|<|\pi|^{\frac{1}{q+1}}\}$. It can be seen by looking at the coordinate function $\phi_0$ and $\phi_1$ of $\Phi$ and observe that the function $\phi_1(u)-w\phi_(u)$ is distinguish of degree 1 for $u\in \{|u|<|\pi|^{\frac{1}{q+1}}\}$ and for a fixed $w\in \Phi(\{|u|<|\pi|^{\frac{1}{q+1}}\})$. For the explicit formula of $\phi_0$ and $\phi_1$, see \cite{GH} sec 25.\\

When $n=1$, either $\gamma=|\pi|^{\frac{1}{q+1}}$ or $|\pi|^{\frac{1}{q+1}}<\gamma<|\pi|^{\frac{1}{q^2+q}}$.\\
If $\gamma=|\pi|^{\frac{1}{q+1}}$ and $|u_1|=\gamma$, then there is a $|u_0|\leq |\pi|^{\frac{q}{q+1}}$ s.t. $u_1$ correspond to $u_0$ of height 1. Let us change our notation and use $u_{1,1}$ for $u_1$ and $u_{0,1}$ for $u_0$. Let $u_{1,2},\dots,u_{1,q+1}$ be the elements that also correspond to $u_{0,1}$. We will show that the equalities $|u_{1,1}-u_{1,i}|=|\pi|^{\frac{1}{q+1}}$ hold for $ i \neq 1$.\\
Define $\psi(x)=\phi_0(u_{1,1})\phi_1(x)-\phi_1(u_{1,1})\phi_0(x)\in \mathbb{C}_p[[x]]$ for $|x|\leq |\pi|^{\frac{1}{q+1}}$.\\
Then
\begin{enumerate}
	\item $\psi(x)$ is distinguish of degree $q+1$ for $|x|\leq |\pi|^{\frac{1}{q+1}}$ and the corresponding coefficient has norm $|\pi|^{\frac{-q}{q+1}}$.\\
	\item $\{u_{1,i}\}$ are simple roots of $\psi$ in $X_\gamma$.\\
	\item $|\frac{\partial\psi}{\partial x}(u_{1,1})|=|\epsilon(u_{1,1})|=1$. see \cite{GH} sec 25.\\ 
\end{enumerate}
So $|\prod_{i\neq 1}(u_{1,i}-u_{1,1})||\pi|^{\frac{-q}{q+1}}=1$. However, $|\prod_{i\neq 1}(u_{1,i}-u_{1,1})|\leq |\pi|^{\frac{q}{q+1}}$ by Lemma \ref{LeS1} (2). Therefore, we can conclude that $|u_{1,i}-u_{1,1}|=|\pi|^{\frac{1}{q+1}}$ for all $i\neq 1$.\\

If $|\pi|^{\frac{1}{q+1}}<\gamma<|\pi|^{\frac{1}{q^2+q}}$, then there exists $u_0$ s.t. $|\pi|^{\frac{q}{q+1}}<|u_0|=\gamma^q<|\pi|^{\frac{1}{q+1}}$ and $u_1$ corresponds to $u_0$ of height 1. Let $u_1,u_2,\dots,u_{q+1}$ correspond to $u_0$ with $|u_{q+1}|$ is smallest. Define $u_{1,i}=u_i$ for $i\neq q+1$, $u_{0,1}=u_0$ and $v_{0,1}=u_{q+1}$. In particular, $|v_{0,1}|=\frac{|\pi|}{|u_0|}$ and $|\pi|^{\frac{q}{q+1}}<|v_{0,1}|<|\pi|^{\frac{1}{q+1}}$.\\
Define $\psi(x)=\phi_0(u_{1,1})\phi_1(x)-\phi_1(u_{1,1})\phi_0(x)\in \mathbb{C}_p[[x]]$ for $|x|\leq \gamma$.\\
Then
\begin{enumerate}
	\item $\psi(x)$ is distinguish of degree $q+1$ for $|x|\leq \gamma$ and the corresponding coefficient has norm $|\pi|^{-1}\gamma$.\\
	\item $\{u_{1,i}\}_{i=1,\dots,q}\cup\{v_{0,1}\}$ are simple roots of $\psi$ in $X_\gamma$.\\
	\item $|\frac{\partial\psi}{\partial x}(u_{1,1})|=|\epsilon(u_{1,1})|=1$.\\ 
\end{enumerate}
So $|\prod^{q}_{i=2}(u_{1,i}-u_{1,1})||v_{0,1}-u_{1,1}||\pi|^{-1}\gamma=1$, implies
$|\prod^{q}_{i=2}(u_{1,i}-u_{1,1})|=|\pi|\gamma^{-2}$. Meanwhile, $|\prod^{q}_{i=2}(u_{1,i}-u_{1,1})|\leq |\pi|\gamma^{-2}$ by Lemma \ref{LeS1} (1). So we conclude $|u_{1,1}-u_{1,i}|=|\pi|^{\frac{1}{q-1}}\gamma^{\frac{-2}{q-1}}$ for each $i\neq 1$.\\

Now assume Theorem \ref{ThmInj} holds for $\gamma<|\pi|^{\frac{1}{q^n+q^{n+1}}}$ for some $n\geq 1$.\\
We need to show that the statement holds for $|\pi|^{\frac{1}{q^n+q^{n+1}}} \leq \gamma <|\pi|^{\frac{1}{q^{n+1}+q^{n+2}}}$.\\

Suppose $\gamma=|\pi|^{\frac{1}{q^n+q^{n+1}}}$. Fix a $u_1\in \partial X_{\gamma}$. Denote $u_{n+1,1}=u_1$. There exists $u_{n,1}$ corresponds to $u_{n+1,1}$ with norm $\gamma^q$. There exists $u_{n-1,1}$ corresponds to $u_{n,1}$ with norm $\gamma^{q^2}$ and so on. Then we have $u_{l,1}$ for $l=0,1,\dots,n+1$ with $|u_{0,1}|\leq |\pi|^{\frac{q}{q+1}}$ and $|u_{l,1}|=\gamma^{q^{n+1-l}}=|\pi|^{\frac{1}{q^l+q^{l-1}}}$ for $l> 0$. Set $u_{0,2}=u_{0,1}$for convenience. Further more, there exists $u_{l,m}$ for $1\leq l\leq n+1$ and $1\leq m \leq q^l+q^{l-1}$ such that $|u_{l,m}|=|u_{l,1}|$ and $u_{l,m}$ corresponds to $u_{l-1,\left\lfloor \frac{m-1}{q}\right\rfloor+1}$ where $\left\lfloor z \right\rfloor$ denote the integral part of $z$. Then $u_{l,m}$ has same image under $\Phi$ as $u_{n+1,1}$ if $n+1-l$ is even.\\

Here, we want to show $|u_{n+1,1}-u_{n+1,m}|=(\frac{|\pi|}{\gamma^{2q^{j-1}}})^{\frac{1}{q^{j}-q^{j-1}}}$ for $1\leq j \leq n $ and $q^{j-1}<m\leq q^{j}$, or equivalent, both $u_{n+1,m}$ and $u_{n+1,1}$ correspond to $u_{n+1-j, 1}$ of smallest possible height $j$. And want to show $|u_{n+1,1}-u_{n+1,m}|=\gamma$ for $m>q^n$.\\

\begin{figure}[htbp]
	\centering
		\includegraphics[scale=0.50]{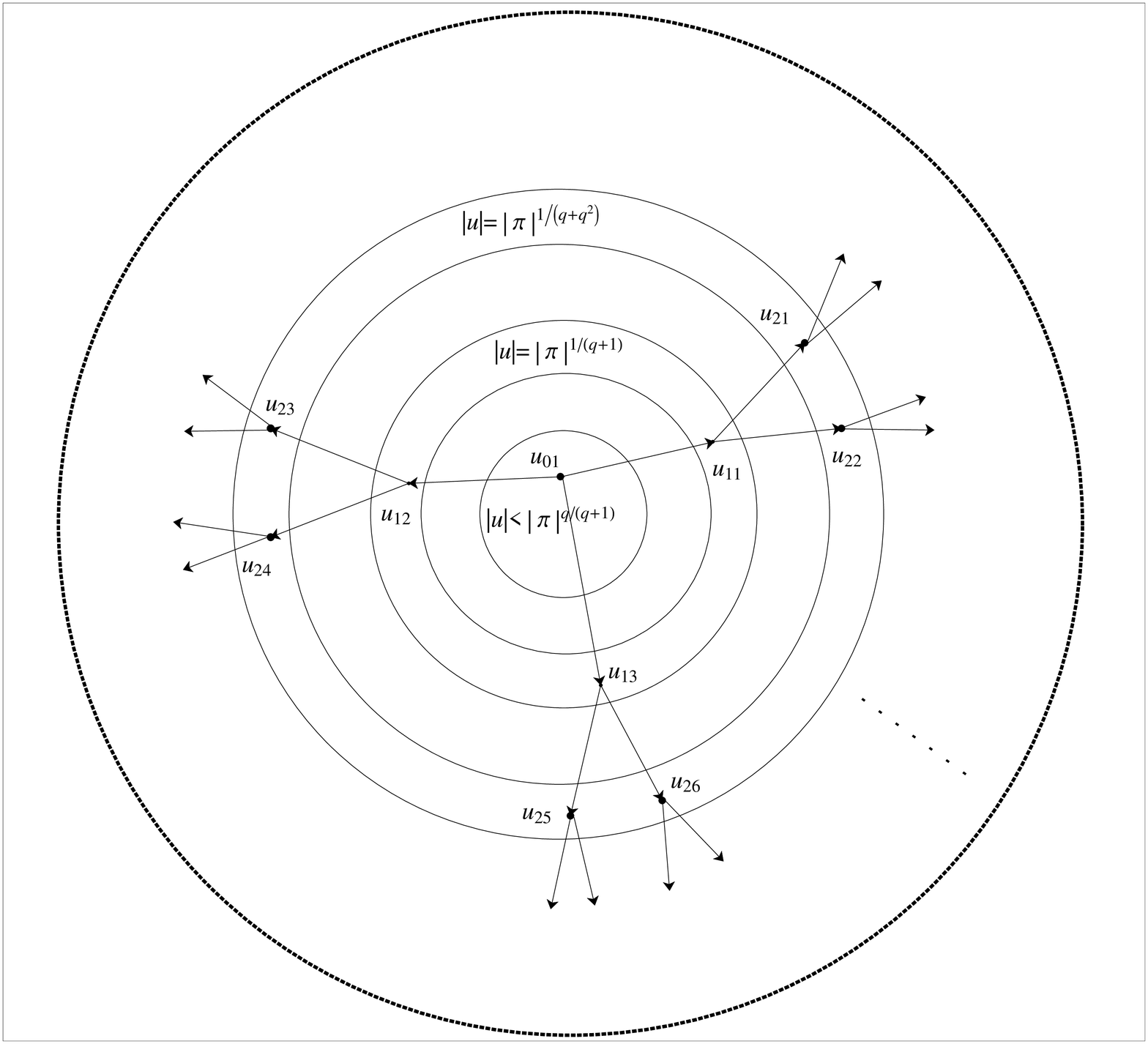}
	\caption{Illustration when $q=2$ and $|u_{01}|\leq|\pi|^{\frac{q}{q+1}}$}
	\label{fig:c1}
\end{figure}

Define $\psi(x)=\phi_0(u_{n+1,1})\phi_1(x)-\phi_1(u_{n+1,1})\phi_0(x)$.\\
Then
\begin{enumerate}
	\item $\psi(x)$ is distinguish of degree $q^{n+1}+q^n+\dots+1$ for $|x|\leq |\pi|^{\frac{1}{q^{n}+q^{n+1}}}$ and the corresponding coefficient has norm $|\pi|^{-n-1}\gamma^{q^n+q^{n-1}+\dots+1}$.\\
	\item $\{u_{l,m}\}_{n+1-l\,\,\text{is even}}$ are simple roots of $\psi$ in $X_\gamma$.\\
	\item $|\frac{\partial\psi}{\partial x}(u_{n+1,1})|=|\epsilon(u_{n+1,1})|=1$.\\ 
\end{enumerate}

So 
\[
\begin{array}{l}
\displaystyle{\left|\prod_{i=1 }^{\left\lfloor \frac{n+1}{2}\right\rfloor}\left(\prod_{m=1 }^{q^{n-2i}+q^{n-2i+1}}(u_{n+1,1}-u_{n+1-2i,m})\right)\right|\left|\prod^{q^{n}+q^{n+1}}_{m=2}(u_{n+1,1}-u_{n+1,m})\right|} \\
\displaystyle{=\left|\pi\right|^{n+1}\gamma^{-q^n-q^{n-1}-\dots-1}}.\\
\end{array}
\]

Since $|(u_{n+1,1}-u_{n+1-2i,m})|=\gamma$ for $i>0$, the above equality becomes
\[\left|\prod^{q^{n}+q^{n+1}}_{m=2}(u_{n+1,1}-u_{n+1-2i,m})\right|=|\pi|^{n+1}\gamma^{-q^{n}-2q^{n-1}-2q^{n-2}-\dots-2}.\] 

On the other hand, from the induction hypothesis, we get $|u_{n,1}-u_{n,m}|=(\frac{|\pi|}{(\gamma^q)^{2q^{j-1}}})^{\frac{1}{q^{j}-q^{j-1}}}$ for $1\leq j\leq n-1 $ and $q^{j-1}<m\leq q^{j}$ and $|u_{n,1}-u_{n,m}|=\gamma^q$ for $m>q^{n-1}$. Apply Lemma \ref{LeCoHe} and Lemma \ref{LeS1}(2) to $\{u_{n,m}\}$, then we get
$|u_{n+1,1}-u_{n+1,m}|\leq (\frac{|\pi|}{\gamma^{2q^{j-1}}})^{\frac{1}{q^{j}-q^{j-1}}}$ for $1\leq j\leq n $ and $q^{j-1}<m\leq q^{j}$ and $|u_{n+1,1}-u_{n+1,m}|\leq\gamma$ for $m>q^n$. Hence, $|\prod^{q^{n}+q^{n+1}}_{m=2}(u_{n+1,1}-u_{n+1-2i,m})|\leq|\pi|^{n+1}\gamma^{-q^{n}-2q^{n-1}-2q^{n-2}-\dots-2}$. So we conclude $|u_{n+1,1}-u_{n+1,m}|=(\frac{|\pi|}{\gamma^{2q^j}})^{\frac{1}{q^{j+1}-q^{j}}}$ for $1\leq j\leq n $ and $q^{j-1}<m\leq q^{j}$ and $|u_{n+1,1}-u_{n+1,m}|=\gamma$ for $m>q^n$.\\

Suppose $|\pi|^{\frac{1}{q^n+q^{n+1}}}<\gamma < |\pi|^{\frac{1}{q^{n+1}+q^{n+2}}}$ 
Fix a $u_1\in \partial X_{\gamma}$. Denote $u_{n+1,1}=u_1$. There exists $u_{n,1}$ corresponds to $u_{n+1,1}$ with norm $\gamma^q$. There exists $u_{n-1,1}$ corresponds to $u_{n,1}$ with norm $\gamma^{q^2}$ and so on. Then we have $u_{l,1}$ for $l=0,1,\dots,n+1$ with $|u_{l,1}|=\gamma^{q^{n+1-l}}$ for $l\geq 0$. There is an unique  $v_{0,1}$ corresponds to $u_{0,1}$ with $|v_{0,1}|=\frac{|\pi|}{|u_{0,1}|}$. There exists $u_{l,m}$ for $1\leq l\leq n+1$ and $1\leq m \leq q^l$  such that $|u_{l,m}|=|u_{l,1}|$ and $u_{l,m}$ corresponds to $u_{l-1,\left\lfloor \frac{m-1}{q}\right\rfloor+1}$. And there exists $v_{l,m}$ for $1\leq l\leq n$ and $1\leq m \leq q^l$ such that $|v_{l,m}|=|v_{0,1}|^{\frac{1}{q^l}}$ and $v_{l,m}$ corresponds to $v_{l-1,\left\lfloor \frac{m-1}{q}\right\rfloor+1}$. Notice that $|u_{l,m}|<|u_{n+1,1}|$ as well as $|v_{l,m}|<|u_{n+1,1}|$ for $l<n+1$. Here $u_{l,m}$ as well as $v_{l-1,m}$ has the same image as $u_{n+1,1}$ under $\Phi$ if $n+1-l$ is even.\\

\begin{figure}[htbp]
	\centering
		\includegraphics[scale=0.50]{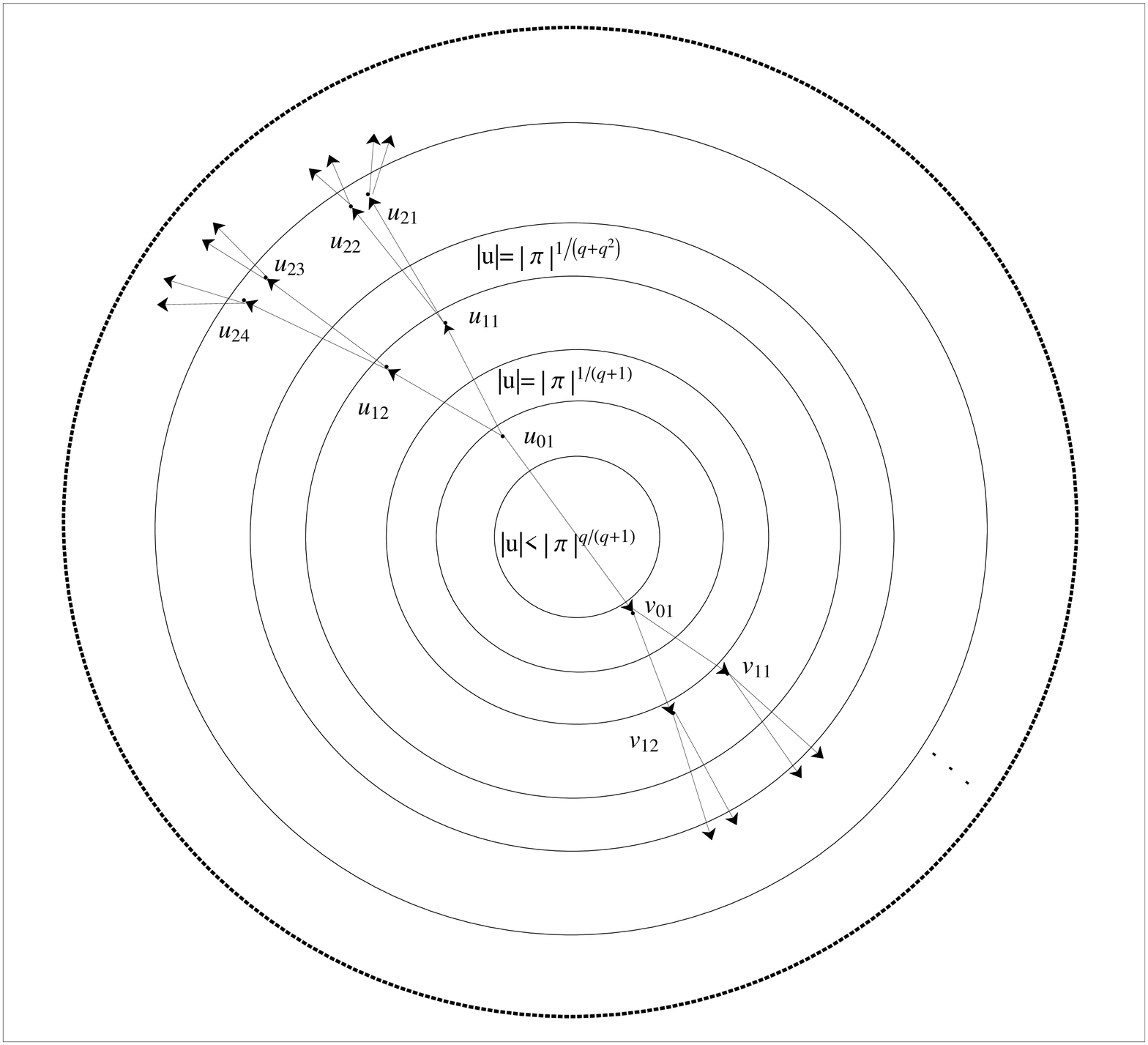}
	\caption{Illustration when $q=2$ and $|\pi|^{\frac{q}{q+1}}<|u_{01}|<|\pi|^{\frac{1}{q+1}}$}
	\label{fig:c2}
\end{figure}

We need to show $|u_{n+1,1}-u_{n+1,m}|=(\frac{|\pi|}{\gamma^{2q^{j-1}}})^{\frac{1}{q^{j}-q^{j-1}}}$ for $1\leq j\leq n+1 $ and $q^{j-1}<m\leq q^{j}$.\\

Define $\psi(x)=\phi_0(u_{n+1,1})\phi_1(x)-\phi_1(u_{n+1,1})\phi_0(x)$.\\
Then
\begin{enumerate}
	\item $\psi(x)$ is distinguish of degree $q^{n+1}+q^n+\dots+1$ for $|x|\leq |\pi|^{\frac{1}{q^{n}+q^{n+1}}}$ and the corresponding coefficient has norm $|\pi|^{-n-1}\gamma^{q^n+q^{n-1}+\dots+1}$.\\
	\item $\{u_{l,m}\}_{n+1-l\,\,\text{is even}}\coprod\{v_{l-1,m}\}_{n+1-l\,\,\text{is even}}$ are simple roots of $\psi$ in $X_\gamma$.\\
	\item $|\frac{\partial\psi}{\partial x}(u_{n+1,1})|=|\epsilon(u_{n+1,1})|=1$.\\ 
\end{enumerate}
So 
\[\begin{array}{l}
\displaystyle{\left|\prod_{ i=1 }^{\left\lfloor \frac{n+1}{2}\right\rfloor}\left(\prod_{m=1}^{q^{n+1-2i}}(u_{n+1,1}-u_{n+1-2i,m})\right)\right|\left|\prod_{m=2}^{q^{n+1}}(u_{n+1,1}-u_{n+1,m})\right|} \\
\displaystyle{\times\left|\prod_{ i=0}^{ \left\lfloor \frac{n}{2}\right\rfloor}\left(\prod_{ m=1}^{ q^{n-2i}}(u_{n+1,1}-v_{n-2i,m})\right)\right|} \\
\displaystyle{=|\pi|^{n+1}\gamma^{-q^n-q^{n-1}-\dots-1}}.\\
\end{array}\]

Hence, $\left|\prod_{m=2}^{q^n+1}(u_{n+1,1}-u_{n+1,m})\right|=|\pi|^{n+1}\gamma^{-2q^n-2q^{n-1}-\dots-2}$.\\

But the induction hypothesis implies $\left|\prod_{m=2}^{q^n+1}(u_{n+1,1}-u_{n+1,m})\right|\leq|\pi|^{n+1}\gamma^{-2q^n-2q^{n-1}-\dots-2}$. Therefore, we conclude $|u_{n+1,1}-u_{n+1,m}|=(\frac{|\pi|}{\gamma^{2q^{j-1}}})^{\frac{1}{q^{j}-q^{j-1}}}$ for $1\leq j\leq n+1 $ and $q^{j-1}<m\leq q^{j}$.\\

\end{proof}

\begin{cor}
The equality of Lemma \ref{LeS1} holds and Lemma \ref{LeCoHe} 1(a) holds for unique $i$.
\end{cor}

\begin{cor}
$\Phi$ is injective on $X^\circ_{|\pi|^{\frac{1}{q+1}}}$. Also, for $\gamma\geq |\pi|^{\frac{1}{q+1}}$ and any $u_0\in \partial X_{\gamma}$, $\Phi$ is injective on $\{u\in X\,\,:\,\,|u-u_0|<|\pi|^{\frac{1}{q-1}}\gamma^{\frac{-2}{q-1}}\}$. We call $\{u\in X\,\,:\,\,|u-u_0|<|\pi|^{\frac{1}{q-1}}\gamma^{\frac{-2}{q-1}}\}$  the domain of injectivity around $u_0$ with $|u_0|\geq |\pi|^{\frac{1}{q+1}}$.
\end{cor}
\begin{proof}
The dominating terms of $\phi_0$ on $X^\circ_{|\pi|^{\frac{1}{q+1}}}$ is the constant term and the dominating terms of $\phi_1$ on $X^\circ_{|\pi|^{\frac{1}{q+1}}}$ is the linear term, so $\frac{\phi_1}{\phi_0}$ is dominated by the linear term and hence is injective on $X^\circ_{|\pi|^{\frac{1}{q+1}}}$.\\

For other cases, observe that if $\gamma\geq |\pi|^{\frac{1}{q+1}}$, then $\{u\in X\,\,:\,\,|u-u_0|<|\pi|^{\frac{1}{q-1}}\gamma^{\frac{-2}{q-1}}\}\cap \Phi^{-1}\left(\Phi(u_1)\right)$ has either 1 or 0 points if $|u_0|=|u_1|=\gamma$. Hence the result follows.
\end{proof}

\begin{rem}
In a previous version of this paper, the proof of Theorem \ref{ThmInj} is done by considering the Newton's Polygon of the function $\psi(x)=\phi_0(u_0)\phi_1(u_0+x)-\phi_1(u_0)\phi_0(u_0+x)$.
\end{rem}

\end{para}

\section{Images of the Domains of Injectivity}

Suppose $|u_0|=\gamma$. And let $[z:w]$ be the homogeneous coordinates on $\bbP^1$.\\

If $ |\pi|^{\frac{1}{q^{s}+q^{s-1}}}<\gamma<|\pi|^{\frac{1}{q^{s+2}+q^{s+1}}}$ for some odd $s$, then $\gamma$ is not a critical radius of $\phi_0$, then the image of domain of injectivity around $u_0$ under $\Phi$ lies in $\{[1:w] \in \bbP^1\}$.  Indeed, for $|u-u_0|<|\pi|^{\frac{1}{q-1}}\gamma^{\frac{-2}{q-1}}$\\

 $$\begin{array}{rcl}
\left|\frac{\phi_1(u)}{\phi_0(u)}-\frac{\phi_1(u_0)}{\phi_0(u_0)}\right| &=& \left|\left(\frac{\phi_1}{\phi_0}\right)'(u_0)\right||u-u_0|\\
& &\\
 &= &  \frac{1}{\|\phi_0\|_{\partial X_\gamma}^2}|u-u_0|\\
 & &\\
 &=& |\pi|^{s+1}\gamma^{-2(1+q+\dots+q^{s})}|u-u_0|.\\
\end{array}$$
Then $\Phi$ is a bijection from $\{\,\,u\,\,:\,\,|u-u_0|<(|\pi|\gamma^{-2})^{\frac{1}{q-1}}\,\,\}$ to $$\left\{\,\,w\,\,:\,\,\left|w-\frac{\phi_1(u_0)}{\phi_0(u_0)}\right|<|\pi|^{s+1}\gamma^{-2(1+q+\dots+q^{s})}(|\pi|\gamma^{-2})^{\frac{1}{q-1}}=|\pi|^{s+1+\frac{1}{q-1}}\gamma^{-2\frac{q^{s+1}}{q-1}}\,\,\right\}\;.$$

In paricular, if $\gamma=|\pi|^{\frac{1}{q^{s+1}+q^{s}}}$ with $s$ is odd, then $|\pi|^{s+1+\frac{1}{q-1}}\gamma^{-2\frac{q^{s+1}}{q-1}}=|\pi|^{s+1-\frac{1}{q+1}}$.\\

Similarly, if $\gamma=|\pi|^{\frac{1}{q^{s+1}+q^{s}}}$ for some even $s$, then $\frac{\phi_0}{\phi_1}$ is a bijective map from \\ $\{\,\,u\,\,:\,\,|u-u_0|<(|\pi|\gamma^{-2})^{\frac{1}{q-1}}\,\,\}$ to \\
$$\left\{\,\,z\,\,:\,\,\left|z-\frac{\phi_0(u_0)}{\phi_1(u_0)}\right|<|\pi|^{s}\gamma^{-2(1+q+\dots+q^{s})}(|\pi|\gamma^{-2})^{\frac{1}{q-1}}=|\pi|^{s-\frac{1}{q+1}}\,\,\right\}.$$

Although $(|\pi|\gamma^{-2})^{\frac{1}{q-1}}\searrow  \left| \pi \right|^{\frac{1}{q-1}}$, hence it is bounded below as $s\rightarrow \infty$, the radius of the image disc goes to zero as $\gamma \rightarrow 1$.

\section{Domains of Analyticity}
\subsection{Local analyticity of the group action}
\begin{para}
Suppose $g\in G=End(\bar{F}\otimes \mathbb{F}_{q^2})\cong \fro^{*}_D$ with
$g=1+\pi^n\xi$ with $\xi\in \mu_{q^2-1}$ and $n\geq 1$. \\

In \cite{Ch}, Chai uses the action of $g$ on the Cartier-Dieudonne module $M_u$ of $F_u $ to compute the action on $g$ on $A_2[[u]]/(\pi)=\mathbb{F}_{q^2}[[u]]$. $M_u$ is a left module generated by an element $e_u$ over Cartier ring $Cart_A(\mathbb{F}_{q^2}[[u]]) $ and $e_u$ is annihilated by $F-V-\left\langle u\right\rangle\in Cart_A(\mathbb{F}_{q^2}[[u]]) $.

Suppose $s=g\cdot u\in \mathbb{F}_{q^2}[[u]]$. Then $g$ extends to a homomorphism from $Cart_A(\mathbb{F}_{q^2}[[u]]) e_s$ to $Cart_A(\mathbb{F}_{q^2}[[u]]) e_u$ which sends $e_s$ to $e_u+\pi^n\xi e_u+\sum_{m\geq 0} V^m\left\langle a_m(u)\right\rangle e_u$ for some $a_m(u)\in \mathbb{F}_{q^2}[[u]]$. So $e_u+\pi^n\xi e_u+\sum_{m\geq 0} V^m\left\langle a_m(u)\right\rangle e_u$ and $e_s$ are annihilated by $F-V-\left\langle s\right\rangle$ and $e_u$ is annihilated by $F-V-\left\langle u\right\rangle$. From this, \cite{Ch} Theorem 1 part (2) conclude the followings congruence relation are true:
\begin{numequation}\label{eq:yu1}  \left\{\begin{array}{rcl}
	a_m(u) &\in& A_2[[u]], \forall m\\
	a_m(u)&\equiv& 0, \forall m\geq n\\
	a_{n-1}(u)&\equiv& (\xi^{q^{n+1}}-\xi^{q^{n}})u^{q^n+q^{n-1}+\dots+q+1}\\
	a_{m}(u) & \equiv &-u^{q^{m+1}}a_{m+1}(u), 0\leq m\leq n-2\\
	u&\equiv& s(1+a_0(u))\\
\end{array}\right. \end{numequation}
The above hold modulo $(\pi)+u^{q-p+(q/p)(q^n+\dots+q+1)}$ when $q>p$ and hold modulo $(\pi)+(u^{p+2(p^n+\dots+p)})+(u^{p}a_0(u))$ when $q=p$.  It is clear that 
\[(\pi)+(u^{q-p+(q/p)(q^n+\dots+q+1)})\subset (\pi)+(u^{2q^n+2q^{n-1}+2q^{n-2}+\dots+2q^2+2q+3})\]
for $q>p$ but we have to be careful for the case when $q=p$. In any case, we can deduce from (\ref{eq:yu1}) that $a_0(u)\not\equiv 0$ modulo $(\pi)+(u^{q^n+2q^{n-1}+2q^{n-2}+\dots+2q^2+2q+2})$ and hence
\[(\pi)+(u^{p+2(p^n+\dots+p)})+(u^{p}a_0(u))\subset (\pi)+(u^{q^n+2q^{n-1}+2q^{n-2}+\dots+2q^2+2q+3}).\]
Therefore, modulo $(\pi)+(u^{q^n+2q^{n-1}+2q^{n-2}+\dots+2q^2+2q+3})$,\\
\[a_0(u)\equiv (-1)^{n-1}(\xi^{q^{n+1}}-\xi^{q^{n}})u^{q^n+2q^{n-1}+2q^{n-2}+\dots+2q^2+2q+1}\]
and 
\[\begin{array}{rcl}
s &\equiv & u(1+(-1)^{n}(\xi^{q^{n+1}}-\xi^{q^{n}})u^{q^n+2q^{n-1}+2q^{n-2}+\dots+2q^2+2q+1}) \\
&\equiv & u+(-1)^{n}(\xi^{q^{n+1}}-\xi^{q^{n}})u^{q^n+2q^{n-1}+2q^{n-2}+\dots+2q^2+2q+2} \\
&\equiv & u+(\xi-\xi^q)u^{q^n+2q^{n-1}+2q^{n-2}+\dots+2q^2+2q+2}.\\
\end{array}\]

Similarly, when $g=1+\pi_D\pi^n\xi$ with $\xi\in \mu_{q^2-1}$ and $n\geq 0$.
$g\cdot u=s$ where $s$ satisfies the relation\\
\begin{numequation} \label{eq:yu2}  \left\{\begin{array}{rcl}
	a_m(u) &\in& A_2[[u]], \forall m\\
	a_m(u)&\equiv& 0, \forall m\geq n+1\\
	a_{n}(u)&\equiv& \xi^{q^{n+2}}u^{q^n+q^{n-1}+\dots+q+1}\\
	a_{m}(u) & \equiv &-u^{q^{m+1}}a_{m+1}(u), 0\leq m\leq n-2\\
	u&\equiv& s(1+a_0(u))\\
\end{array}\right. \end{numequation}
hold modulo $(\pi)+(u^{q-p+(q/p)(q^n+\dots+q+1)})$ when $q>p$ and hold modulo $(\pi)+(u^{p+2(p^n+\dots+p)})+(u^{p}a_0(u))$ when $q=p$. We can deduce that $s\equiv u$ modulo $(\pi)+(u^{2q^n+2q^{n-1}+2q^{n-2}+\dots+2q^2+2q+2})$.\\

\begin{prop}
Suppose $n>0$. The group $1+\pi^n\fro_D$ acts local analytically on $\Omega_n=\{u\in X(\mathbb{C}_p):|u|<|\pi|^{\frac{1}{q^n+q^{n+1}}}\}$.
\end{prop}
\begin{proof}
We know that for $u_0\in \Omega_n$, the period map $\Phi$ is bijective on $\Delta_{u_o}=\{u: |u-u_0|<|\pi u^{-2}_o|^{\frac{1}{q-1}}\}$. And $1+\pi^n\fro_D$ acts analytically on $\Phi(\Delta_{u_0})$. Since $\Phi$ is $\fro^{*}_D$-equivariant, we only need to prove that for $g\in 1+\pi^n\fro_D$, $|g\cdot u_0-u_0|<|\pi u^{-2}_o|^{\frac{1}{q-1}}\} $.\\
But the argument in \cite{Ch} can show easily that $1+\pi^{n+1}\fro_D$ acts trivially on $u$ modulo $(\pi)+(u^{q^{n+1}+q^n+\dots+q+1})$. So the result state above enough to conclude that $|g\cdot u_0-u_0|\leq \max\{|\pi|,|u_0|^{q^n+2q^{n-1}+2q^{n-2}+\dots+2q^2+2q+2}\}< |\pi u^{-2}_o|^{\frac{1}{q-1}}$ for $|u_0|<|\pi|^{\frac{1}{q^n+q^{n+1}}}$ and $g\in 1+\pi^n\fro_D$.
\end{proof}
A similar statement for $1+\pi_D\pi^n\fro_D$.\\
\begin{prop}
Suppose $n\geq 0$. The group $1+\pi_D\pi^n\fro_D$ acts local analytically on $\Omega_n:=\{u\in X(\mathbb{C}_p):|u|<|\pi|^{\frac{1}{2q^n}}\}$. 
\end{prop}

\end{para}

\nocite{*}

\bibliographystyle{plain}

\bibliography{ref}

\end{document}